\documentclass[11pt,a4paper]{amsart}

\usepackage[margin=1.25in,a4paper]{geometry}
\usepackage{lmodern}
\usepackage[utf8]{inputenc}
\usepackage[T1]{fontenc}

\usepackage{microtype}
\usepackage[shortlabels]{enumitem}
\usepackage{svg}

\usepackage[english]{babel}
\usepackage{mathrsfs}
\usepackage{amssymb}

\usepackage{xcolor}
\definecolor{linkcol}{RGB}{0,80,158}
\definecolor{citecol}{RGB}{46,117,120}

\usepackage[colorlinks]{hyperref}
\hypersetup{
    colorlinks=true,
    linkcolor=linkcol,
    urlcolor=linkcol,
    citecolor=citecol,
}
\usepackage[hyperpageref]{backref}
\usepackage[msc-links,alphabetic]{amsrefs}
\usepackage[capitalize,nameinlink,noabbrev]{cleveref}
\crefname{equation}{}{}
\numberwithin{equation}{section}

\theoremstyle{plain}
\newtheorem{theorem}{Theorem}
\newtheorem{lemma}[theorem]{Lemma}

\theoremstyle{definition}

\theoremstyle{remark}
\newtheorem*{remark}{Remark}

\newcommand{\bC}{\mathbb{C}}

\newcommand{\bQ}{\mathbb{Q}}

\newcommand{\bT}{\mathbb{T}}
\newcommand{\bZ}{\mathbb{Z}}

\newcommand{\sH}{\mathscr{H}}

\newcommand{\abs}[1]{\lvert#1\rvert}

\newcommand{\norm}[1]{\lVert#1\rVert}

\renewcommand{\Re}{\operatorname{Re}}
\renewcommand{\Im}{\operatorname{Im}}
\newcommand{\dd}{\mathrm{d}}

\title{On the asymptotic growth of the norm of the partial sum operator on $\sH^1$}

\author{Viktor Andersson}
\address{Department of Mathematical Sciences, Norwegian University of Science and Technology (NTNU), 7491 Trondheim, Norway}
\email{viktor.andersson@ntnu.no}
\date{\today}

\begin{document}
\begin{abstract}
We show that the norm of the partial sum operator $S_N$ on the Hardy space $\sH^1$ of Dirichlet series satisfies
$$\liminf_{N\to\infty}\frac{\log\log N}{\log N}\norm{S_N}_{\sH^1\to\sH^1}\geq\frac{1-\log 2}{2\pi}.$$
Combined with an upper bound of Bondarenko, Brevig, Saksman and Seip, this shows that
$$\norm{S_N}_{\sH^1\to\sH^1}\asymp\frac{\log N}{\log\log N}$$
as $N\to\infty$.
\end{abstract}
\thanks{The author is supported by Grant 354537 of the Research Council of Norway.}
\maketitle

\section{Introduction}\label{section:introduction}

For $0<p<\infty$, the Hardy space of Dirichlet series $\sH^p$ is defined as the completion of the set of Dirichlet polynomials $f(s)=\sum_{n=1}^N a_nn^{-s}$ in the norm (or quasi-norm when $0<p<1$)
$$\norm{f}_{\sH^p}=\left(\lim_{T\to\infty}\frac{1}{2T}\int_{-T}^T\abs{f(it)}^p\,\dd t\right)^{1/p},$$
and for $p=\infty$, the space $\sH^\infty$ is defined as the space of bounded analytic functions $f$ in the open right half-plane $\Re s>0$ that can be represented by a convergent Dirichlet series in some right half-plane, equipped with the supremum norm. The elements of $\sH^p$ are convergent Dirichlet series in the half-plane $\Re s>1/2$ for $0<p<\infty$ and $\Re s>0$ for $p=\infty$, and as such it is natural to consider the partial sum operator $S_N$, defined by mapping
$$S_N:\sum_{n=1}^\infty a_nn^{-s}\mapsto\sum_{n=1}^Na_nn^{-s},$$
on these spaces. Although the norm of these operators is difficult to compute explicitly unless $p=2$, an interesting line of research has been to determine the asymptotics of $\norm{S_N}_{\sH^p\to\sH^p}$ as $N\to\infty$ for different values of $0<p\leq\infty$. In the range $1<p<\infty$, Konyagin, Queffélec, Saksman and Seip \cite{konyagin_riesz_2022}*{Theorem 5.1} showed that
$$\sup_{N\geq1}\norm{S_N}_{\sH^p\to\sH^p}=\frac{1}{\sin(\pi/p)},\quad 1<p<\infty,$$
which corresponds to the norm of the one-dimensional Riesz projection $P^+:L^p(\bT)\to H^p(\bT)$. In the quasi-Banach range $0<p<1$, it was shown by Bondarenko, Brevig, Saksman and Seip \cite{bondarenko_linear_2019}*{Theorem 5.3} that
\begin{equation}\label{eq:norm-asymptotic-quasi-range}
\log\norm{S_N}_{\sH^p\to\sH^p}\asymp\frac{\log N}{\log\log N},\quad 0<p<1
\end{equation}
as $N\to\infty$. Recently, the asymptotic order for $p=\infty$ was also settled: it holds that
$$\norm{S_N}_{\sH^\infty\to\sH^\infty}\asymp\log N$$
as $N\to\infty$. The upper bound was established already in 2006 by Balasubramanian, Calado and Queffélec \cite{balasubramanian_bohr_2006}*{Lemma 1.1} using an effective version of Perron's formula, and the lower bound was established in 2026 by Chen and Fu \cite{chen_note_2026}*{Theorem 1.1}. In the Banach space range $1\leq p\leq\infty$, this leaves only the case of $p=1$, and to the author's knowledge, the best known bounds are from 2019 due to Bondarenko, Brevig, Saksman and Seip \cite{bondarenko_linear_2019}*{Theorem 5.2}, and say that
$$\log\log N\lesssim\norm{S_N}_{\sH^1\to\sH^1}\lesssim\frac{\log N}{\log\log N}$$
as $N\to\infty$. The goal of this short paper is to show that the upper bound gives the correct asymptotic order by showing the corresponding lower bound, thus settling the last remaining asymptotic order for the partial sum on the Hardy spaces of Dirichlet series in the Banach space range.

\begin{theorem}\label{theorem:partial-sum}
    It holds that
    $$\liminf_{N\to\infty}\frac{\log\log N}{\log N}\norm{S_N}_{\sH^1\to\sH^1}\geq\frac{1-\log 2}{2\pi}.$$
    Consequently,
    $$\norm{S_N}_{\sH^1\to\sH^1}\asymp\frac{\log N}{\log\log N}$$
    as $N\to\infty$.
\end{theorem}

The idea of the proof is similar to the idea used by Bondarenko, Brevig, Saksman and Seip to prove the lower bound in \cref{eq:norm-asymptotic-quasi-range}, namely to first establish a lower bound on the norm $\norm{S_N}_{\sH^1\to\sH^1}$ when $N=p_1\cdots p_d$ is the product of the first $d$ primes by testing on suitable functions, and then extending this to all $N$ by appropriately multiplying by well-chosen Dirichlet monomials. Where our argument differs is that to establish the lower bound for $N=p_1\cdots p_d$, we employ a simple modification of an inequality of Helson \cite{helson_conjugate_1959}.

\section{The proof}

Let us start by reminding the reader of the connection between the spaces $\sH^p$ and the Hardy spaces of the infinite-dimensional torus $H^p(\bT^\infty)$, as this will play a central role in the proof. For this we will restrict ourselves to the Banach space range $1\leq p\leq\infty$. The infinite-dimensional torus $\bT^\infty=\bT\times\bT\times\cdots$ is a compact abelian group, and so it has a unique normalized Haar measure $m_\infty$, which is the product measure of the normalized arclength measure $m$ on each circle $\bT$. We can identify the dual group of $\bT^\infty$ with the positive rationals $\bQ^+$ by identifying $q\in\bQ^+$ with the map
$$\chi\mapsto\chi(q)=\chi_1^{k_1}\cdots\chi_n^{k_n}$$
with $\chi=(\chi_1,\chi_2,\dots)\in\bT^\infty$, where $q=p_1^{k_1}\cdots p_n^{k_n}$ is the unique prime factorization of $q$ and $p_1,p_2,\dots$ the increasing sequence of primes. In particular, any $f\in L^p(\bT^\infty)$ can be identified with its Fourier series
$$f(\chi)\sim\sum_{q\in\bQ^+}\widehat f(q)\chi(q),$$
where
$$\widehat f(q)=\int_{\bT^\infty}f(\chi)\overline{\chi(q)}\,\dd m_\infty(\chi)$$
are the Fourier coefficients of $f$. The Hardy spaces of $\bT^\infty$ are then defined as
$$H^p(\bT^\infty)=\{f\in L^p(\bT^\infty):\widehat f(q)=0\text{ for all }q\in\bQ^+\setminus\bZ^+\}.$$
We can identify $\sH^p$ with $H^p(\bT^\infty)$ through the \emph{Bohr lift}, which is the isometric isomorphism that maps the Dirichlet series
$$f(s)=\sum_{n=1}^\infty a_nn^{-s}$$
in $\sH^p$ to the function $f^*\in H^p(\bT^\infty)$ with Fourier series
$$f^*(\chi)\sim\sum_{n=1}^\infty a_n\chi(n).$$
The reader is referred to \cite{queffelec_diophantine_2020}*{Chapter 6} for more details. We say that a function $T$ is a \emph{trigonometric polynomial} on $\bT^\infty$ if it is of the form
$$T(\chi)=\sum_{q\in Q}a_q\chi(q)$$
for some finite set $Q\subseteq\bQ^+$, and we say that it is an \emph{analytic polynomial} on $\bT^\infty$ if, in addition, $Q\subseteq\bZ^+$. Under the identification through the Bohr lift, we will instead work with the partial sum on $H^1(\bT^\infty)$ in our proof, which then takes $f\in H^1(\bT^\infty)$ to the analytic polynomial
$$S_Nf(\chi)=\sum_{n=1}^N\widehat f(n)\chi(n).$$

We are now ready to start tackling the proof. For this, we will need the following small modification of a theorem of Helson \cite{helson_conjugate_1959}*{Theorem 2}.

\begin{lemma}\label{lemma:entropy-inequality}
    Let $T$ be a trigonometric polynomial on $\bT^\infty$ with $\widehat T(q)=0$ for all rational numbers $q>1$, and suppose that $R=\Re T\geq0$. Then
    $$\int_{\bT^\infty}R(\chi)\log R(\chi)\,\dd m_\infty(\chi)\leq\widehat R(1)\log\widehat R(1)+\frac{\pi}{2}\norm{\Im T}_{L^1(\bT^\infty)}$$
    with the convention that $0\log 0=0$.
\end{lemma}

\begin{proof}
    The Fourier coefficients of the trigonometric polynomial $T-\widehat T(1)$ are supported in $(0,1)\cap\bQ^+$ by assumption, and so, as $(0,1)\cap\bQ^+$ is closed under multiplication, the same holds for the Fourier coefficients of $(T-\widehat T(1))^n$ for all positive integers $n$. From this, together with the binomial theorem, one readily deduces that
    $$\int_{\bT^\infty}p(T(\chi))\,\dd m_\infty(\chi)=p(\widehat T(1))$$
    for all polynomials $p:\bC\to\bC$. By applying Runge's approximation theorem (e.g., \cite{conway_functions_1978}*{Chapter VIII, Theorem 1.7}) to the analytic function $z\mapsto (z+\varepsilon)\log(z+\varepsilon)$ in a compact convex subset of the closed right half-plane containing the image of $T$ and $\widehat T(1)$, it then also follows that
    $$\int_{\bT^\infty}(T(\chi)+\varepsilon)\log(T(\chi)+\varepsilon)\,\dd m_\infty(\chi)=(\widehat T(1)+\varepsilon)\log(\widehat T(1)+\varepsilon)$$
    for all $\varepsilon>0$. Taking real parts, using that $\abs{T+\varepsilon}\geq R+\varepsilon$, that $\abs{\arg (T(\chi)+\varepsilon)}<\pi/2$, and the inequality $\Re(z\log z)\leq(\Re z)\log(\Re z)$, it follows that
    $$\int_{\bT^\infty}(R(\chi)+\varepsilon)\log(R(\chi)+\varepsilon)\,\dd m_\infty(\chi)\leq(\widehat R(1)+\varepsilon)\log(\widehat R(1)+\varepsilon)+\frac{\pi}{2}\norm{\Im T}_{L^1(\bT^\infty)}.$$
    The result follows by letting $\varepsilon\downarrow0$.
\end{proof}

\begin{remark}
The above lemma and the result of Helson can be seen as a quantitative and infinite-dimensional version of a result originally due to M. Riesz on conjugate functions on the circle (see, for example, \cite{garnett_bounded_1981}*{Chapter 3, Exercise 7} or \cite{zygmund_trigonometric_1968}*{Chapter VII, Theorem 2.10}). One can also easily extend the lemma to the setting of compact abelian groups with ordered duals with essentially the same proof; compare with \cite{rudin_fourier_1962}*{Chapter 8}. Viewed from this perspective, the importance of our assumption that $\widehat T(q)=0$ for all rational numbers $q>1$ comes from ordering $\bQ^+$ by $-\log q$, in which case $\bQ^+\cap(1,\infty)$ corresponds to the negative Fourier spectrum.
\end{remark}

With this we can now prove \cref{theorem:partial-sum}.

\begin{proof}[Proof of \cref{theorem:partial-sum}]
    Consider the sequence $\{f_d\}_{d\in\bZ^+}$ of analytic polynomials on $\bT^\infty$ defined by
    $$f_d(\chi)=\prod_{j=1}^d\frac{(1+\chi_j)^2}{2}=\chi_1\cdots\chi_dR_d(\chi),$$
    where
    $$R_d(\chi)=\prod_{j=1}^d(1+\Re\chi_j).$$
    One easily sees that $\norm{f_d}_{L^1(\bT^\infty)}=1$. Write $N_d=p_1\cdots p_d$, and note that
    \begin{equation}\label{eq:n-asymptotic}
        d=(1+o(1))\frac{\log N_d}{\log\log N_d}
    \end{equation}
    as $d\to\infty$ by the prime number theorem (see, e.g., \cite{montgomery_multiplicative_2007}*{proof of Theorem 2.10} for the details) and that $\chi(N_d)=\chi_1\cdots\chi_d$. By expanding $f_d$ in terms of the Fourier coefficients of $R_d$, we see also that
    $$S_{N_d}f_d(\chi)=\chi(N_d)\sum_{q\in\bQ^+\cap(0,1]}\widehat R_d(q)\chi(q).$$
    We therefore consider the trigonometric polynomial
    $$T_d(\chi)=2\overline{\chi(N_d)}S_{N_d}f_d(\chi)-1=1+2\sum_{q\in\bQ^+\cap(0,1)}\widehat R_d(q)\chi(q).$$
    Using that $\overline{\widehat R_d(q)}=\widehat R_d(1/q)$ since $R_d$ is real-valued, that $\widehat R_d(1)=1$, and that $\overline{\chi(q)}=\chi(1/q)$, one readily uses this to compute
    $$\Re T_d(\chi)=\sum_{q\in\bQ^+}\widehat R_d(q)\chi(q)=R_d(\chi).$$
    Using that $\abs{\chi(N_d)}=1$ and \cref{lemma:entropy-inequality}, we can then estimate
    $$\norm{S_{N_d}f_d}_{L^1(\bT^\infty)}\geq\frac{1}{2}\norm{\Im T_d}_{L^1(\bT^\infty)}\geq\frac{1}{\pi}\int_{\bT^\infty}R_d(\chi)\log R_d(\chi)\,\dd m_\infty(\chi)=\frac{1-\log2}{\pi}d,$$
    where the last equality follows from a simple computation using Fubini's theorem and the identity
    $$\frac{1}{2\pi}\int_0^{2\pi}(1+\cos\theta)\log(1+\cos\theta)\,\dd\theta=1-\log2.$$
    Fix now an arbitrary integer $N\geq6$, and set
    $$d_0=\max\{d\in\bZ^+:N/N_d\geq N_d+1\}.$$
    Note that $\log N_{d_0}=\left(\frac{1}{2}+o(1)\right)\log N$ as $N\to\infty$, which, in combination with \cref{eq:n-asymptotic}, yields that
    $$d_0=\left(\frac{1}{2}+o(1)\right)\frac{\log N}{\log\log N}$$
    as $N\to\infty$. Set
    $$g_N(\chi)=\chi(\lfloor N/N_{d_0}\rfloor)f_{d_0}(\chi)$$
    and observe that then
    $$S_Ng_N(\chi)=\chi(\lfloor N/N_{d_0}\rfloor) S_{N_{d_0}}f_{d_0}(\chi),$$
    which is easily seen by expanding both sides in terms of their Fourier coefficients and using that an integer $n$ that is divisible by $\lfloor N/N_{d_0}\rfloor$ satisfies $1\leq n\leq N$ if and only if it satisfies $1\leq n\leq N_{d_0}\lfloor N/N_{d_0}\rfloor$. Consequently
    $$\norm{S_Ng_N}_{L^1(\bT^\infty)}=\norm{S_{N_{d_0}}f_{d_0}}_{L^1(\bT^\infty)}\geq\frac{1-\log2}{\pi}d_0=\left(\frac{1-\log2}{2\pi}+o(1)\right)\frac{\log N}{\log\log N}$$
    as $N\to\infty$. As $\norm{g_N}_{L^1(\bT^\infty)}=\norm{f_{d_0}}_{L^1(\bT^\infty)}=1$ and $g_N\in H^1(\bT^\infty)$ (it is an analytic polynomial), the first part of the theorem follows. The second part follows immediately by combining this with the upper bound in \cite{bondarenko_linear_2019}*{Theorem 5.2}.
\end{proof}

\subsection*{Acknowledgments.} I would like to thank my PhD supervisor, Ole Fredrik Brevig, for his insightful comments that helped improve the exposition in the paper.

\subsection*{AI Declaration.}
The initial proof of \cref{theorem:partial-sum}, including \cref{lemma:entropy-inequality}, was found by OpenAI's GPT-6 Astra after the author asked it about the asymptotic order of the partial sum on $\sH^1$. This proof was subsequently modified by the author to improve presentation and readability; however, the core of the argument remains the same. The final version of the paper is entirely written by the author, and the author takes full responsibility for its content.

\bibliography{references}

\end{document}